\documentclass[11pt]{amsart}
\usepackage{amsmath,amssymb, graphicx, amscd,latexsym,here}
\makeatletter
\newtheorem{Theorem}{Theorem}
\newtheorem{MainTheorem}[Theorem]{Main Theorem}
\newtheorem{Lemma}[Theorem]{Lemma}
\newtheorem{Corollary}[Theorem]{Corollary}
\newtheorem{Proposition}[Theorem]{Proposition}

\newtheorem{Remark}[Theorem]{Remark}

\newtheorem{Assertion}[Theorem]{Assertion}

\newcommand{\eps}{\varepsilon}

\newcommand\la{\lambda}
\newcommand\vphi{\varphi}

\newcommand\al{\alpha}

\newcommand\be{\beta}

\newcommand\ga{\gamma}
\newcommand\Ga{\Gamma}
\newcommand\de{\delta}
\newcommand\De{\Delta}

\newcommand\BC{ {\mathbb C}}
\newcommand\BN{ {\mathbb  N}}
\newcommand\BQ{ {\mathbb  Q}}
\newcommand\BZ{{\mathbb  Z}}
\newcommand\BR{ {\mathbb  R}}
\newcommand\BS{ {\mathbb  S}}

\newcommand\bfq{\mbox {\bf  q}}

\newcommand\bfu{\mbox {\bf  u}}

\newcommand\bfv{\mbox {\bf  v}}

\newcommand\bfx{\mbox {\bf  x}}

\newcommand\bfw{\mbox {\bf  w}}

\newcommand\bfz{\mbox {\bf  z}}

\newcommand\bfa{\mbox {\bf  a}}
\newcommand\bfb{\mbox {\bf  b}}

\newcommand\gtr{\bigtriangledown}

\newcommand\nl{\newline}

\newcommand\grad{\rm{grad}\/}

\newcommand\Ker{\rm{Ker}\/}

\newcommand\rdeg{{\rm{rdeg}\/}}
\newcommand\pdeg{{\rm{pdeg}\/}}

\newcommand\inv{^{-1}}

\def\mapright#1{\smash{\mathop{\longrightarrow}\limits^{{#1}}}}

\def\mapdown#1{\Big\downarrow\rlap{$\vcenter{\hbox{$#1$}}$}}

\def\inv{^{-1}}

\begin{document}
\title[Contact structure on mixed links
]
{   Contact structure on mixed links}

\author
[M. Oka ]
{Mutsuo Oka }
\address{\vtop{
\hbox{Department of Mathematics}
\hbox{Tokyo  University of Science}
\hbox{1-3 Kagurazaka, Shinjuku-ku}
\hbox{Tokyo 162-8601}
\hbox{\rm{E-mail}: {\rm oka@rs.kagu.tus.ac.jp}}
}}

\begin{abstract}
A strongly non-degenerate mixed function has a  Milnor  open book
 structures on a sufficiently small sphere. We introduce the notion of
 {\em a holomorphic-like} mixed function
and  we will show that  a link defined by such a mixed  function has a canonical contact structure.
Then we will show that this contact structure for a certain holomorphic-like mixed function
 is carried by the Milnor open book. 
\end{abstract}
\maketitle

\maketitle

\section{Introduction}
Let $f(\bfz)$ a holomorphic function with an isolated critical point at
the origin.
Then the Milnor fibration of $f$ carries a canonical contact structure (\cite{Giroux, C-N-PP}).
We consider a similar problem for mixed functions $f(\bfz,\bar \bfz)$.
We have shown that strongly non-degenerate mixed  functions have Milnor fibrations on a small sphere
\cite{OkaMix}. 
However the situation is very different in the point  that  the tangent space of  a mixed hypersurface is not  a complex vector space.
Therefore the  restriction of the canonical contact structure need not give a contact structure
on the mixed link. We introduce a class of mixed functions called
{\em holomorphic-like}
and we show that the restriction of the canonical contact structure gives a contact structure on the link (Theorem \ref{holomorphic-like}).
A typical  class of mixed functions we consider are given as the pull-back  $g(\bfw,\bar\bfw)=\vphi_{a,b}^* f(\bfw,\bar\bfw)$
of  a convenient non-degenerate holomorphic function $f(\bfz)$ by 
a homogeneous mixed covering $\vphi_{a,b}:\BC^n\to \BC^n$
which is 
defined by 
$\vphi_{a,b}(\bfw,\bar\bfw)=(w_1^a\bar w_1^b,\dots, w_n^a\bar w_n^b)$.
Such a pull-back is a typical example of a holomorphic-like mixed function.
Then we will show also that  the Milnor open book is compatible with the
canonical  contact structures
for these mixed functions
 (Theorem \ref{MainTheorem2}).
 
 I would like to thank M. Ishikawa and V. Blanloeil for the  stimulating 
 discussions.

\section{Preliminaries}
\subsection{Mixed functions and polar weightedness}
Consider complex analytic  function of $2n$-variables  $F(z_1,\dots,z_n,w_1,\dots, w_n)$
expanded in a convergent series $\sum_{\nu,\mu}c_{\nu,\mu}
\bfz^\nu{\bfw}^\mu$ and consider the restriction $f(\bfz,\bar\bfz)$
which is defined by the 
substitution $w_j=\bar z_j,\,j=1,\dots,n$.
We  call this real analytic function {\em an analytic  mixed function}.
Namely
$f(\bfz,\bar \bfz)=\sum_{\nu,\mu}c_{\nu,\mu}\bfz^\nu\bar\bfz^\mu$
where $\bfz=(z_1,\dots,z_n)$, $\bar \bfz=(\bar z_1,\dots,\bar z_n)$,
$ \bfz^\nu=z_1^{\nu_1}\cdots z_n^{\nu_n}$ for $ \nu=(\nu_1,\dots,\nu_n)$
(respectively
$\bar\bfz^\mu=\bar z_1^{\mu_1}\cdots\bar z_n^{\mu_n}$  for 
$ \mu=(\mu_1,\dots,\mu_n)$).
 Here $\bar z_i$ is the complex conjugate of $z_i$.  Assume that $f$ is a polynomial.
Writing $z_j=x_{j}+ i \, y_{j}$, it is easy to see that $f$ is a
polynomial of $2n$-variables $x_1,y_1,\dots, x_{n},y_n$.
 In this case, we  call $f$ {\em a mixed polynomial}
of
$z_1,\dots, z_n$.

A mixed polynomial
$f(\bfz,\bar \bfz)$ is called {\em polar weighted homogeneous}
if there exist   positive  integers
$q_1,\dots, q_n$ and $p_1,\dots, p_n$ and non-zero integers $m_r,\,m_p$
such that 
\begin{eqnarray*}
&\gcd(q_1,\dots,q_n)=1,\quad\gcd(p_1,\dots, p_n)=1,\, \\
&\sum_{j=1}^n q_j( \nu_j+\mu_j)=m_r,\quad
\sum_{j=1}^n p_j (\nu_j-\mu_j)=m_p,\quad \text{if}\,\, c_{\nu,\mu}\ne 0
\end{eqnarray*} 
The weight vectors $Q=(q_1,\dots, q_n)$ and $P=(p_1,\dots, p_n)$ are called {\em the radial weight and the polar weight } respectively.
Using  radial weight and the polar weight, we define
the radial $\BR_{>0}$-action  and the polar $S^1$-action as follows.
\begin{eqnarray*}
& r\circ \bfz=(r^{q_1} z_1,\dots, 
r^{q_n} z_n),\quad r\in \BR_{>0}\\
e^{i\eta}\circ \bfz &=(e^{ip_1\eta}z_1,\dots, 
e^{i p_n\eta}z_n),\quad e^{i\eta}\in S^1
\end{eqnarray*}
Then 
$f$ satisfies the
functional equalities
\begin{eqnarray}\label{polar-weight}
f(r \circ (\bfz,\bar \bfz) )&=r^{m_r}f(\bfz,\bar\bfz),\, \, r \in \BR_{>0}\\
f(e^{i\eta}\circ (\bfz,\bar \bfz) )&=e^{i m_p\eta}f(\bfz,\bar\bfz),\,\,e^{i\eta}\in S^1.
\end{eqnarray}
These equalities give the following Euler equalities.
\begin{eqnarray}
&\text{(Radial Euler equality)}:\,\,m_r f(\bfz,\bar\bfz)=\sum_{i=1}^n q_i(\frac{\partial f}{\partial
 z_i}z_i+\frac{\partial f}{\partial
 \bar z_i}\bar z_i
)\\
&\text{(Polar Euler equality)}:\,\,m_p f(\bfz,\bar\bfz)=\sum_{i=1}^n p_i(\frac{\partial f}{\partial
 z_i}z_i-\frac{\partial f}{\partial
 \bar z_i}\bar z_i
).
\end{eqnarray}
We consider a special type of polar weighted homogeneous polynomial.
A polar weighted homogeneous polynomial $f(\bfz,\bar\bfz)$ is called 
{\em strongly polar weighted homogeneous}   if the radial weight and the polar weights are the same,
{\it i.e.},  $p_j=q_j$ for $j=1,\dots, n$.
In this case, the radial and polar Euler equalities gives:
\begin{eqnarray}\label{strongEuler}
\begin{cases}
&\sum_{j=1}^np_jz_j\frac{\partial f}{\partial z_j}(\bfz,\bar \bfz)=\frac{m_r+m_p}2 f(\bfz,\bar \bfz),\\
& \sum_{j=1}^np_j\bar z_j\frac{\partial f}{\partial\bar z_j}(\bfz,\bar \bfz)=\frac{m_r-m_p}2 f(\bfz,\bar \bfz)
\end{cases}
\end{eqnarray}
The above equalities say that $f(\bfz,\bar\bfz)$ is a weighted homogeneous
polynomial for $\bfz$ and $\bar\bfz$ independently.
Furthermore  $f(\bfz,\bar\bfz)$ is called 
{\em strongly polar positive weighted homogeneous}   if $\pdeg f=m_p>0$.
\subsection{Euclidean metric and hermitian product}
 Recall that $\BC^n$ is canonically identified with $\BR^{2n}$ by 
 $\bfz=(z_1,\dots, z_n)\mapsto \bfz_{\BR}:=(x_1,y_1,\dots, x_n,y_n)\in\BR^{2n}$.
 The inner product in $\BR^{2n}$ is simply the real part of the hermitian product in $\BC^n$.
 We denote the hermitian inner product  as $(\bfz,\bfw)$ for $\bfz,\bfw\in \BC^n$
  and the inner product as the vector in 
 $\BR^{2n}$ as $(\bfz_{\BR},\bfw_{\BR})_{\BR}$.
Namely putting $\bfw=(w_1,\dots, w_n) $ with 
$w_j=u_j+iv_j$,
\[
 (\bfz,\bfw)=\sum_{j=1}^n z_j\bar w_j,\,\,\,
(\bfz_{\BR},\bfw_{\BR})_{\BR}=\sum_{j=1}^n(x_ju_j+y_jv_j).
\] Thus  $   \Re (\bfz,\bfw)=(\bfz_{\BR},\bfw_{\BR})_{\BR}$.
By the triviality of the tangent bundle $T_p\BR^{2n}$, we identify
$T_p\BR^{2n}$ with $\BR^{2n}=\BC^n$.
Thus $\sum_{j=1}^n  (x_j(\frac{\partial}{\partial
x_j})_p+y_j(\frac{\partial}{\partial y_j})_p)$
is identified with 
 $\bfz=(z_1,\dots, z_n)\in \BC^n$ and $z_j=x_j+i\,y_j$.
 Recall  the complexified  tangent vectors are defined by
\begin{eqnarray*}
&\frac{\partial}{\partial z_j}=\frac 12 \left(
\frac{\partial}{\partial x_j}-i\frac{\partial}{\partial y_j}\right),
\quad\frac{\partial}{\partial \bar z_j}=\frac 12 \left(
\frac{\partial}{\partial x_j}+i\frac{\partial}{\partial y_j}
\right)\end{eqnarray*}
Thus a complex vector $\bfz\in \BC^{2n}=\BR^{2n}$ is identified with the 
tangent vector 
\[\sum_{j=1}^n (x_j\frac {\partial}{\partial x_j}+y_j\frac{\partial}{\partial y_j})=\sum_{j=1}^n (z_j\frac{\partial}{\partial z_j}+\bar z_j \frac{\partial}{\partial \bar z_j}).
\]
$J:T\BC^n\to T \BC^n$ is the almost complex structure defined by
\[\begin{split}
J(\frac{\partial}{\partial x_j})=\frac{\partial}{\partial y_j},\quad
J(\frac{\partial}{\partial y_j})=-\frac{\partial}{\partial x_j}\\
J(\frac{\partial}{\partial z_j})=i\frac{\partial}{\partial z_j},\quad
J(\frac{\partial}{\partial \bar z_j})=-i\frac{\partial}{\partial \bar z_j}.
\end{split}
\]
 For a real valued mixed function 
 $h(\bfz,\bar \bfz)$, we define {\em the  real gradient  } $\grad_{\BR} h\in \BR^{2n}$
 (or {\em  Riemannian gradient} in \cite{C-N-PP}) 
 as
 \begin{eqnarray*}
 &{\grad _{\BR}}\, h(\bfz,\bar \bfz)=\left(\cfrac{\partial h}{\partial
  x_1}(\bfz,\bar \bfz),\cfrac{\partial h}{\partial y_1}(\bfz,\bar
  \bfz)\dots, \cfrac{\partial h}{\partial x_{n}}(\bfz,\bar
  \bfz),\cfrac{\partial h}{\partial y_n}(\bfz,\bar \bfz)\right)\in \BR^{2n}.
 \end{eqnarray*}
 We define also {\em the complex gradient  of $h$} 
 or {\em hermitian gradient} in \cite{C-N-PP}  as follows.
 \[
\bigtriangledown\,h(\bfz,\bar \bfz)=2\left(\frac{\overline{\partial h}}{\partial z_1}(\bfz,\bar\bfz),
 \dots, \frac{\overline{\partial h}}{\partial
 z_n}(\bfz,\bar\bfz)\right)\in \BC^n.
 \]

 \begin{Proposition} Assume that $h(\bfz,\bar\bfz)$ is a real valued mixed function and
 let $\bfz(t)=(z_1(t),\dots, z_n(t)),\, z_j(t)=x_j(t)+ i\, y_j(t),\,-1\le t\le 1$ be a smooth curve in $\BC^n$
 and let  $\bfz(0)=\bfu$,
 $\frac{d\bfz}{dt}(0)=\bfv$.
 Then we have
 \begin{eqnarray*}
 \frac{dh(\bfz(t),\bar\bfz(t))}{dt}(0)=
 ({\bfv}_{\BR},\grad_{\BR}h(\bfu,\bar\bfu))_{\BR}
 &=   \Re(\bfv,{\bigtriangledown}\,h(\bfu,\bar\bfu)).
 \end{eqnarray*}
 \end{Proposition}
 \begin{proof}
 The second equality follows from the simple calculation:
 \[\begin{split}
 \frac{dh(\bfz(t),\bar\bfz(t))}{dt}(0)&=\sum_{i=1}^n v_i
 \frac{\partial h}{\partial z_i}(\bfz_0,\bar\bfz_0)+
 \sum_{i=1}^n \bar v_i \frac{\partial h}{\partial \bar z_i}(\bfz_0,\bar\bfz_0)\\
 &=\sum_{i=1}^n v_i \frac{\partial h}{\partial z_i}(\bfz_0,\bar\bfz_0)+
\sum_{i=1}^n \bar v_i \overline{\frac{\partial h}{\partial z_i}}
(\bfz_0,\bar\bfz_0)\\
 &=   \Re (\bfv,\bigtriangledown h(\bfz_0,\bar\bfz_0)).
 \end{split}
 \]
 \end{proof}\noindent
 Thus  the tangent space of the  real hypersurface $H:=h\inv(0)$ at a smooth point $\bfz_0$
 is given by
 \[\begin{split}
 T_{\bfx_0}H&=\{\bfu\in \BR^{2n}\,|\, (\bfu,{\grad_{\BR}} h(\bfz_0,\bar \bfz_0))_{\BR}=0\}\\
 &=\{\bfw\in \BC^n\,|\,    \Re(\bfw,{ \bigtriangledown\,} h(\bfz_0,\bar\bfz_0))=0\}.
 \end{split}
 \]
For our  later purpose, it is more convenient to use the hermitian 
gradient.
\subsubsection{Holomorphic function case.}
Assume that $f(\bfz)$ is a holomorphic function defined on a neighborhood of the origin.
Then the hermitian gradient 
$\bigtriangledown f$ is defined  by (see \cite{Milnor,C-N-PP})
\[
\bigtriangledown f(\bfz)=(\overline{\frac{\partial f}{\partial z_1}},\dots,
\overline{\frac{\partial f}{\partial z_n}}).
\]
Consider the complex hypersurface $V=f\inv(0)\subset \BC^n$.
\begin{Proposition}
Assume that $p\in V$ is a regular point. Then the tangent space
$T_pV$ is the complex subspace given by
\[
T_pV=\{\bfv\in \BC^n\,|\, (\bfv,\bigtriangledown f(\bfz))=0\}.
\]
\end{Proposition}
\begin{Remark}
Let $f(\bfz,\bar\bfz)$ be a  complex valued mixed  function and write
$f(\bfz,\bar\bfz)=g(\bfz,\bar\bfz)+i\,h(\bfz,\bar\bfz)$. Consider a mixed hypersurface $V=f\inv(0)$
and assume that $p\in V$ is a mixed regular point.
Then the tangent space $T_pV$ has no complex structure in general and there does not exist
a single gradient vector to describe $T_pV$.
It is described by two hermitian  gradient vectors as
\[
T_pV=\{\bfv\in \BC^n\,|\,
\Re (\bfv,\bigtriangledown g(p,\bar p))=\Re(\bfv,\bigtriangledown h(p,\bar p))=0\}.
\]
\end{Remark}
 \subsubsection{Weighted spheres}
 For a given positive integer vector $\bfa=(a_1,\dots, a_n)$ and
a positive number $r>0$, we consider 
 \[
 \rho_{\bfa}(\bfz)=\sum_{j=1}^n a_j|z_j|^2=\sum_{j=1}^{n} a_j(x_{j}^2+y_j^2)
 \]
 and we define {\em the weighted sphere } ${\BS}_r(\bfa)$ by
  \[{\BS_r}(\bfa):=\{\bfz\in \BC^n\,|\, \rho_{\bfa}(\bfz)=r^2\}.\]
  The standard sphere is  defined  by the weight vector $\bfa=(1,\dots,1)$ and 
  in this case, we simply write ${\BS_r}$.
 Put $\tilde \bfz(\bfa)=(a_1z_1,\dots, a_n z_n)$.
   Then 
 $\bigtriangledown\, \rho_{\bfa}(\bfz)=2 \tilde \bfz(\bfa)$. 
  Therefore the tangent space at $\bfz_0\in {\BS_r}(\bfa)$  is given by
 \[\begin{split}
 T_{\bfz_0}{\BS}_{r}(\bfa) =\{\bfw\,|\,    \Re (\bfw,\tilde \bfz_0(\bfa))=0\},\quad
 T_{\bfz_0}{\BS}_{r}=\{\bfw\,|\,    \Re (\bfw,\bfz_0)=0\}.
 \end{split}
 \]
\subsubsection{Transversality of a polar weighted homogeneous
   hypersurface.}
Let $f$ be a polar weighted homogeneous polynomial
of radial weight type \nl
$(q_1,\dots, q_n;m_r)$
and of polar weight type $(p_1,\dots, p_n;m_p)$. 
Let $V=f\inv(0)$ and write
$f(\bfz,\bar\bfz)=h(\bfz,\bar\bfz)+ig(\bfz,\bar\bfz)$ with real valued
mixed functions
$h,g$.
\begin{Proposition} (Transversality)\label{Transversality}
Assume that $V$ has an isolated mixed singularity at the origin.
Then 
 the
 sphere
 ${\BS_r}(\bfa)=\{\bfz\in\BC^n; \rho_{\bfa}(\bfz)=r^2\}$
intersects transversely with $V$ for any $r>0$.
\end{Proposition}
\begin{proof}
The proof is the exact same as that of Proposition 4, \cite{OkaPolar}.
Assume that $\bfz_0\in S_{\bfa}(r)\cap V$ is a point where the sphere is
 not transverse.
Note that the tangent space is the real orthogonal space to
two hermitian gradient vectors
$\bigtriangledown\,h(\bfz_0,\bar \bfz_0)$ and $\bigtriangledown\,g(\bfz_0,\bar \bfz_0)$.
Let $\rho_{\bfa}(\bfz,\bar\bfz)=\sum_{j=1}^n a_j|z_j|^2$.
The non-transversality implies  for example, there is a linear relation
\begin{eqnarray}\label{contradiction}
{ \bigtriangledown\,}\rho_{\bfa}(\bfz_0,\bar \bfz_0)=\al\, { \bigtriangledown\,} h(\bfz_0,\bar\bfz_0)+\be\, 
{ \bigtriangledown\,}g(\bfz_0,\bar\bfz_0)
\end{eqnarray}
with some $\al,\be\in
 \BR$.
 We consider the radial orbit curve
 $\bfz(t)=t\circ \bfz_0=(t^{q_1}z_{01},\dots, t^{q_n}z_{0n})$.
 The the tangent vector
 $\frac{dz}{dt}(1)=\tilde \bfz_0(\bfq)=(q_1z_{01},\dots, q_nz_{0n})$
 with $\bfq=(q_1,\dots, q_n)$.
Then we have an inequality:
\begin{eqnarray}\label{positive}
\frac{d\,\rho_{\bfa}(\bfz(t))}{dt}|_{t=1}=   \Re(\tilde \bfz_0(\bfq),{ \bigtriangledown\,}\rho_{\bfa}(\bfz_0,\bar\bfz_0))
=   \Re(\tilde \bfz_0(\bfq),2\tilde \bfz_0(\bfa))>0
\end{eqnarray}
On the other hand, the mixed real polynomials $h(\bfz,\bar \bfz),g((\bfz,\bar \bfz)$ are radially weighted homogeneous under the same weight $\bfq=(q_1,\dots, q_n)$.
This implies $h(\bfz(t))\equiv g(\bfz(t))\equiv  0$ and we have two equalities:
\begin{eqnarray*}
&\cfrac{dh(\bfz(t))}{dt}|_{t=1}=   \Re(\tilde  \bfz_0(\bfq),\bigtriangledown\,h(\bfz_0,\bar \bfz_0))=0,\\
&\cfrac{dg(\bfz(t))}{dt}|_{t=1}=   \Re(\tilde  \bfz_0(\bfq),\bigtriangledown\,g(\bfz_0,\bar \bfz_0))=0.
\end{eqnarray*}
Now we have a contradiction to ( \ref{positive})  by (   \ref{contradiction}):
\begin{eqnarray*}
&&0<  \Re \,(\tilde \bfz_0(\bfq),{ \bigtriangledown\,}\rho_{\bfa}(\bfz_0,\bar\bfz_0))=\\
&&\qquad\qquad \alpha \,  \Re (\tilde \bfz_0(\bfq),{ \bigtriangledown\,}h(\bfz_0,\bar
 \bfz_0))+
\be \,
   \Re(\tilde \bfz_0(\bfq),{ \bigtriangledown\,}g(\bfz_0,\bar \bfz_0))=0.
\end{eqnarray*}
\end{proof}
\subsection{Mixed functions of strongly polar weighted homogeneous face type.}
Consider a mixed function $f(\bfz,\bar\bfz)=\sum_{\nu,\mu}c_{\nu\mu}
\bfz^{\nu}{\bar \bfz}^\mu$. Recall that  for a weight vector $P=(p_1,\dots, p_n)$,
the face function $f_P$ is defined by 
the linear sum of the monomials with the radial degree is the minimum (\cite{OkaMix}).
Thus $f_P(\bfz,\bar\bfz)$ is a radially weighted homogeneous polynomial with the weight $P$.
\subsubsection{Definition.}
$f$ is called {\em a mixed function of polar positive weighted homogeneous face type}
if for any  weight vector $P$ with $\dim\,\De(P)=n-1$,  the face function $f_P(\bfz,\bar \bfz)$
is a polar weighted homogeneous polynomial  with  some  weight
vector $P'$
($P'$ need not be $P$) and $\pdeg_{P'}f_P>0$.

$f$ is called {\em a mixed function of strongly  polar positive  weighted homogeneous face type}
if  the face function $f_P(\bfz,\bar \bfz)$ is a strongly polar  positive
weighted homogeneous polynomial with the same  weight vector $P$, for any $P$ with 
$\dim\, \De(P)=n-1$.
\begin{Proposition}
(1) Assume that $f(\bfz,\bar \bfz)$ is a convenient mixed function 
of polar positive weighted homogeneous face type.
Then for any weight vector $P$, $f_P$ is also polar weighted homogeneous
 polynomial.

(2)
 Assume that $f(\bfz,\bar \bfz)$ is a convenient mixed function 
of strongly polar positive  weighted homogeneous face type.
Then for any weight vector $P$, $f_P$ is also a strongly  polar positive weighted homogeneous
 polynomial.
\end{Proposition}
\begin{proof}
The assertion (1) is obvious, as any face $\De$ of $\Ga(f)$ is a subface
 of a face of dimension $n-1$.
We consider the assertion (2).
The assertion is proved by the descending induction on $\dim\,\De(P)$.
The assertion for the  case $\dim\,\De(P)=n-1$ is the definition itself.
Suppose that $\dim\,\De(P)=k$ and the assertion is true for faces with $\dim\,\De\ge k+1$.
In the dual Newton diagram, $P$ is contained in the interior of a cell
$\Xi$ whose vertices $Q$ satisfies 
$\dim\,\De(Q)\ge k+1$. This implies 
$P$ is a linear combination
$\sum_{j=1}^s a_j\,Q_j$ with $a_j\ge 0$ and $\dim\De(Q_j)\ge k+1$
where $Q_1,\dots, Q_s$ are vertices of $\Xi$.
This   implies also that $\De(P)=\cap_{j}^s\, \De(Q_j)$.
Write $f_P(\bfz,\bar\bfz)=\sum_{k} c_k \bfz^{\nu_k}{\bar\bfz}^{\mu_k}$.
As $f_{Q_j}$ is a strongly polar weighted homogeneous polynomial with weight $Q_j$,
\[
\pdeg_{Q_j} \bfz^{\nu_k}{\bar\bfz}^{\mu_k}=m_j,\, \quad j=1,\dots,s
\]
where $m_j$ is independent of $k$.
This implies 
$f_P$ is polar weighted homogeneous polynomial of weight $P$ with polar degree
$\sum_{j=1}^s a_jm_j>0$. 
\end{proof}
As an obvious but important example, we have
\begin{Proposition}
A holomorphic function $f(\bfz,\bar \bfz)$ is a mixed function of strongly polar
positive weighted homogeneous face type.
\end{Proposition}
A mixed function of strongly polar weighted homogeneous face
 type behaves  like a non-degenerate holomorphic function. In \cite{OkaVar}, we have proved a Varchenko type formula for the zeta function.
\subsection{Mixed cyclic covering.} 
Consider two non-negative integer vectors $\bfa=(a_1,\dots, a_n)$ and $\bfb=(b_1,\dots, b_n)$.
We say $\bfa$ is {\em strictly bigger} than $\bfb$ if  $a_j>b_j\ge 0$ for any $j=1,\dots, n$.
If this is the case, we denote it as $\bfa\gg \bfb$.
For given $\bfa, \bfb$ with $\bfa\gg \bfb$, we consider real analytic mapping  $\vphi_{\bfa,\bfb}$:
\[
\vphi_{\bfa,\bfb}:\,\BC^n\to \BC^n,\quad \vphi_{\bfa,\bfb}(\bfw)=(w_1^{a_1}{\bar w_1}^{b_1},\dots, w_n^{a_n}{\bar w_n}^{b_n}).
\]
We call $\vphi_{\bfa,\bfb}$ a {\em mixed cyclic covering mapping }
associated with
integer vectors $\bfa=(a_1,\dots, a_n)$ and $\bfb=(b_1,\dots, b_n)$.
In fact,  over $\BC^{*n}$, $\vphi_{\bfa,\bfb}:\,\BC^{*n}\to \BC^{*n}$ 
is a $\prod_{j=1}^n (a_j-b_j)$ -fold polycyclic covering.

We say that 
$\vphi_{\bfa,\bfb}$ is  {\em homogeneous } if $\bfa=(a,\dots,a)$ and $\bfb=(b,\dots,b)$
where $a,b$ are integers such that $a>b\ge 0$.
In this case, we denote $\vphi_{a,b}$ instead of $\vphi_{\bfa,\bfb}$
and we call $\vphi_{a,b}$ a homogeneous mixed covering.
For a given mixed function $f(\bfz,\bar\bfz)$, the pull-back
$g=\vphi_{\bfa,\bfb}^*(f)$
is defined by 
$$
g(\bfw,\bar \bfw)=f\circ \vphi_{\bfa,\bfb}(\bfw,\bar\bfw)=f(w_1^{a_1}\bar w_1^{b_1},
\dots, w_n^{a_n}\bar w_{n}^{b_n}).
$$
\begin{Proposition}
Let $f(\bfz, \bar \bfz)$ be a non-degenerate convenient mixed function of 
polar weighted homogeneous face type.
Let $\vphi=\vphi_{\bfa,\bfb}$ be the mixed cyclic covering associated with $\bfa=(a_1,\dots, a_n)$ and 
$\bfb=(b_1,\dots, b_n)$ as above.
Consider the pull-back $g(\bfw,\bar \bfw)=f(\vphi(\bfw,\bar \bfw))$.  Then $g(\bfw,\bar \bfw)$ is a convenient
non-degenerate mixed function of polar weighted homogeneous face type. 

If $f$ is of strongly polar positive weighted homogeneous  face type
and $\vphi=\vphi_{a,b}$ is  a homogeneous mixed covering mapping, $g$ is also of strongly polar positive weighted homogeneous
face type.
\end{Proposition}
\begin{proof}
Let $P$ be a weight vector and consider $f_P(\bfz,\bar\bfz)$.
It is a radially weighted homogeneous polynomial under the weight $P$. Let 
$R=(r_1,\dots, r_n)$ be the polar weight of $f_P$. Let 
$d_r$  and $d_p$ be the  radial and polar degree of $f_P$.
 We consider the normalized  weight $Q=(q_1,\dots, q_n)\in \BQ^n$
 and $S=(s_1,\dots, s_n)\in \BQ^n$ where 
$q_j=p_j/d_r$ and $s_j=r_j/d_p$.  
We consider also the  normalized weights $\hat Q=(\hat q_1,\dots, \hat q_n)$
and $\hat S=(\hat s_1,\dots, \hat s_n)$ where 
\[
\hat q_j=q_j/(a_j+b_j),\,\,\hat s_j=s_j/(a_j-b_j),\quad j=1,\dots, n.
\]
Consider a monomial $M=z_1^{m_1}\bar z_1^{\ell_1}\dots z_n^{m_n}\bar
 z_n^{\ell_n}$ in $f_P$, {\it i.e.} 
$\deg_QM=1,\,\pdeg_S M=1$.
Consider the pull-back of $M$,
\[
M'=\vphi^*M=\prod_{j=1}^n (w_j^{a_j}\bar w_j^{b_j})^{m_j} 
(\bar w_j^{a_j} w_j^{b_j})^{\ell_j} 
\]
Then by an easy calculation, we have
\begin{eqnarray*}
\deg_{\hat Q} M'=\sum_{j=1}^n q_j(m_j+\ell_j)=\deg_Q M=1\\
\pdeg_{\hat S} M'=\sum_{j=1}^n s_j(m_j-\ell_j)=\pdeg_S M=1
\end{eqnarray*}
This implies that  $\vphi^* f_P=g_{\hat Q}$ is a  radially weighted homogeneous
 polynomial
 by the normalized weight  vector
$\hat Q$ and $\vphi^* f_P$ is a polar weighted homogeneous polynomial
by the normal weight vector  $\hat S$. 
Non-degeneracy is the result of the commutative diagram:
\[
\begin{matrix}
\BC^{*n}&\mapright{\vphi}&\BC^{*n}\\
\mapdown{g_{\hat Q}}&&\mapdown{f_Q}\\
\BC&=&\BC
\end{matrix}
\]
We observe that if $\vphi=\vphi_{a,b}$ and $p_j=r_j$,  
\[\hat s_j(a-b)d_p=r_j=p_j=\hat q_j(a+b) d_r.
\]
which implies that $\vphi^*f_P$ is strongly polar weighted homogeneous.
\end{proof}
As holomorphic functions are obviously 
 mixed functions of strongly polar weighted homogeneous face type, we have:
\begin{Corollary}
Assume that $f(\bfz)$ is a convenient non-degenerate holomorphic function
and   $g(\bfw,\bar\bfw)=\vphi^*f(\bfw,\bar\bfw)$
 with $\vphi=\vphi_{\bfa,\bfb}$.
Then $g(\bfw,\bar\bfw)$ is a convenient non-degenerate mixed function of polar weighted homogeneous face  type.
If further $\vphi=\vphi_{a,b}$, homogeneous with $a>b\ge 0$, $g$ is of strongly polar positive  weighted homogeneous face type.
\end{Corollary}

\section{Contact structure}
\subsection{Contact  structure and  a contact submanifold of  a sphere}\label{contact-natural}
Let $M$ be a smooth oriented manifold of dimension $2n-1$.
A  {\em contact structure on $M$ } is a hyperplane distribution $\xi$ in the tangent bundle $TM$
($M\ni p\mapsto \xi(p)\subset T_pM$)
which is induced by a global  1-form $\al$ 
by
$\xi(p)=\Ker\,\al$ such that 
$\al\land(d\al)^{n-1}$ is nowhere vanishing $(2n-1)$ form. We say 
$\al$ is  {\em positive} if $\al\land (d\al)^{n-1}$ is a positive form.


We consider the radius function 
$\rho(\bfz,\bar\bfz)=z_1\bar z_1+\cdots+z_n\bar z_n$. The level manifold 
$\rho\inv(r^2)$ is nothing but the sphere ${\BS}_r$.
On ${\BS}_r$, we consider the canonical contact structure $\xi$ defined by the contact form
$\al:=\,-d^c\rho=\,-d\rho\circ J$
where $J$ is the complex structure.
More explicitly,
\[\al=\sum_{j=1}^n -i(\bar z_j dz_j-z_j d\bar z_j)=2\sum_{j=1}^n(x_jdy_j-y_jdx_j) .
\]
$\xi(\bfz)$ is nothing but the complex
 hyperplane which is hermitian orthogonal to $\bfz$:
$\xi(\bfz)=\{\bfv\,|\, (\bfv,\bfz)=0\}$. 

Let $\omega=d\al=-dd^c\vphi$.
Then $\omega$ is explicitly written as
\[
\omega(\bfz)=2 i\sum_{j=1}^n dz_j\land \bar dz_j=\,4\sum_{j=1}^n dx_j\land dy_j
\]
and $\omega$ defines a symplectic structure on $\xi$.
We have a canonical equality (\cite{C-N-PP}):
\begin{eqnarray}\label{omega-formula}
 4\,   \Re(\bfu,\bfv)=\omega(\bfu,J\bfv),\quad \bfu,\bfv\in T{\BS}_r.\end{eqnarray}
The {\em  Reeb vector field} $R\in \Ga({\BS}_r,T{\BS}_r)$ is defined by 
the property: 
\[
\al(R)=1, \quad
\iota_R(\omega)=0.
\]
Here $\iota_R$ is the inner derivative by $R$. In our case, 
\begin{eqnarray*}
R(\bfz)&=&\frac{ i\bfz}{2 \rho(\bfz)},\,\,\qquad
\text{or as a tangent vector}\qquad\\
&=&\frac i{2 \rho(\bfz)}\sum_{j=1}^n (z_j\frac{\partial}{\partial z_j}-
\bar z_j\frac{\partial}{\partial\bar z_j})=\frac{1}{2r^2}
\sum_{j=1}^n(x_j\frac{\partial}{\partial y_j}
-y_j\frac{\partial}{\partial x_j}).
\end{eqnarray*}

We consider a real codimension two submanifold $K\subset {\BS}_r$.
We say $K$ is {\em a (positive) contact submanifold of ${\BS}_r$} if 
the restriction
$\al_{|K}$ defines a contact submanifold,
{\em i.e.} $(2n-3)$-form $\al\land(d\al)^{n-2}$ is nowhere vanishing form (respectively 
positive form) of $K$.
\subsection{Remarks on the orientation} The orientation of ${\BS}_r$ is given as follows.
A $(2n-1)$-form $\Omega$ is positive if and only if 
$d\rho\land \Omega$ is a positive form of $\BC^n$.
Thus $\al\land (d\al)^{n-1}$ is positive.
Let $f(\bfz,\bar\bfz)=g(\bfz,\bar\bfz)+i\,h(\bfz,\bar\bfz)$ be a non-degenerate mixed function with an
 isolated mixed singularity at the origin.
Let $K_r=f\inv(0)\cap {\BS}_r$ with a sufficiently small $r$.
The orientation of $K_r$ is given by an $(2n-3)$ form $\Omega'$
such that 
$d\rho\land \Omega'\land dg\land dh$ is a positive form of $\BC^n$.

\subsection{Contact structure on  mixed links.}
First we prepare a lemma.
Put $f(\bfz,\bar\bfz)=g(\bfz,\bar\bfz)+ i\,h(\bfz,\bar\bfz)$, where $g,h$ are real valued mixed functions. We use  hereafter the following  notation for simplicity .
\[f_{z_j}=\frac{\partial f}{\partial z_j},\qquad f_{\bar z_j}=\frac{\partial f}{\partial \bar z_j}.
\]
\begin{Lemma}\label{keylemma}
{\rm (1)} $d\rho\land \al$ is given as follows.
\[
d\rho\land \al =i\sum_{a,b=1}^n  A_{a,\bar b}\, dz_a\land d\bar z_b,
\quad
A_{a,\bar b}=2\bar z_a z_b.
\]
{\rm (2)} The two form $dg\land dh$ can be written as follows.
\[
dg\land dh =i\, \sum_{a,b=1}^n B_{a,\bar b}dz_a\land d\bar z_{b}+R
\]
where
\[
B_{a,\bar b}=\frac 12 \left(
f_{z_a} \overline{f_{z_b}} - \overline{f_{\bar z_ a}}  f_{\bar z_b}
\right)
\]
 $R$ is a linear combination of two forms
$dz_{a}\land dz_b$ and $d\bar z_a\land d\bar z_b$. 
\end{Lemma}
\begin{proof}
The assertion (1)  is a result of  a simple calculation:
\begin{eqnarray*}
d\rho\land \al=\left(\sum_{j=1}^n \left(z_jd\bar z_j+\bar z_j dz_j\right)\right)
\land 
\left(i\sum_{k=1}^n z_kd\bar z_k-\bar z_k dz_k\right).
\end{eqnarray*}
For (2), we use the equality 
\[g=\frac12 (f+\bar f), \quad h=\frac {-i}2 (f-\bar f).
\]
Thus we have
\begin{eqnarray*}
dg=\frac 12\sum_{j=1}^n \left\{(f_{z_j}+\bar f_{z_j})dz_j+(f_{\bar z_j}+\bar f_{\bar z_j})d\bar z_j
\right\}\\
dh=\frac{-i}2 \sum_{j=1}^n\left\{
(f_{z_j}-\bar f_{z_j})dz_j+(f_{\bar z_j}-\bar f_{\bar z_j})d\bar z_j)
\right\}
\end{eqnarray*}
As $\bar f_{\bar z_j}=\overline{f_{z_j}}$ and $\bar f_{z_j}=\overline{f_{\bar z_j}}$, the assertion follows by a simple calculation.
\end{proof}
\begin{Corollary} \label{4-form}
The four form
$d\rho\land \al\land dg\land dh$ is given as follows.
\begin{eqnarray*}\label{keycorollary}
&d\rho\land \al\land dg\land dh=-\sum_{a,b=1}^n C_{a,b} dz_a\land d\bar z_a\land dz_b\land
d \bar z_b+S\\
&C_{a,b}=|\bar z_a {f_{z_b}}-\bar z_b { f_{z_a}}|^2 -|z_a f_{\bar z_b}-z_b f_{\bar z_a}|^2
\end{eqnarray*}
where $S$ is a linear combination of other type of four forms.
\end{Corollary}
\begin{proof}
Write 
\[
d\rho\land \al\land dg\land dh=-\sum_{a,b=1}^n C_{a,b} dz_a\land d\bar z_a\land dz_b\land
d \bar z_b+S.
\]
Then by Lemma \ref{keylemma}, we have
\begin{eqnarray*}
C_{a,b}&= A_{a,\bar a} B_{b,\bar b} +A_{b,\bar b} B_{a,\bar a}- A_{a,\bar b} B_{b,\bar a}
-A_{b,\bar a} B_{a,\bar b}\qquad\qquad\qquad\\
&=|z_a|^2(|f_{z_b}|^2-|f_{\bar b}|^2)+|z_b|^2(|f_{z_a}|^2-|f_{\bar z_a}|^2)\qquad\qquad\qquad\\
&-2 \bar z_a z_b(f_{z_b} \overline{f_{ z_a}} - \overline{f_{\bar z_b}}f_{\bar a})
- 2 \bar z_b z_a (f_{z_a} \overline{f_{ z_b}} - \overline{f_{\bar z_a}}f_{\bar b})\\
&=(z_a\overline{f_b}-z_b \overline{f_a})(\bar z_a f_b-\bar z_b f_a)
-(z_b\overline{f_a}-z_a \overline{f_b})(\bar z_b f_a-\bar z_a f_b)\\
&=|\bar z_a {f_{z_b}}-\bar z_b{ f_{z_a}}|^2 -|z_a f_{\bar z_b}-z_b f_{\bar z_a}|^2.\qquad\qquad\qquad\qquad
\end{eqnarray*}
\end{proof}
Define $C(\bfz,\bar\bfz):=\sum_{1\le a< b\le n} C_{a,b}(\bfz,\bar\bfz)$.
By Corollary \ref{4-form} and an easy computation gives the following.
\begin{Corollary}\label{contact-form}
We have 
\begin{eqnarray*}
&&d\rho\land\al\land d\al^{n-2}\land dg\land
   dh(\bfz,\bar\bfz)\\
&=& i^n 2^{n-2}(n-2)!C(\bfz,\bar\bfz) dz_1\land 
\bar z_1\land\cdots\land dz_n\land d \bar z_n\\
&=&4^{n-1} (n-2)!\,C(\bfz,\bar\bfz) dx_1\land dy_1\land\dots \land dx_n\land dy_n.
\end{eqnarray*}
\end{Corollary}
\subsection{Holomorphic-like mixed function.}
Let $U$ be an open neighborhood of the origin.
A mixed function  $f(\bfz,\bar\bfz)$  which is defined on $U$ with an isolated  mixed singularity at the origin is called
{\em holomorphic-like} (respectively {\em anti-holomorphic-like})
 if 
for any $\bfz\in f\inv(0)\cap U$, 
\begin{eqnarray}\label{ineq}
C(\bfz,\bar \bfz)&=&\sum_{1\le a<b\le n}C_{a,b}\ge 0\\
C_{a,b}&=&|z_a\overline{f_{z_b}}-z_b\overline{ f_{z_a}}|^2 -|z_a f_{\bar z_b}-z_b f_{\bar z_a}|^2
\end{eqnarray}
( or respectively 
$C(\bfz,\bar \bfz)=\sum_{1\le a<b\le n}C_{a,b}\le 0$).

We say that  $f(\bfz,\bar\bfz)$ is {\em strictly holomorphic-like}
(resp. {\em strictly  anti-holomorphic-like }) in $U$
if  $C(\bfz,\bar\bfz)>0$ (resp. $C(\bfz,\bar\bfz)<0$)
on any smooth point  $\bfz\in U\cap f\inv(0)\setminus\{\bf0\}$.

\noindent{\bf Remark.}
If $f$ is a holomorphic function, $f_{\bar z_j}=0$ and 
$C_{a,b}\ge 0$ for any $1\le a<b\le n$. Thus
$f(\bfz)$ is obviously (strictly) holomorphic-like.
If $f(\bar \bfz)$ is an anti-holomorphic function,
$f_{z_j}=0$ and 
$f(\bfz)$ is  anti-holomorphic-like.


\begin{Lemma}
Assume that $f(\bfz)$ is a 
 holomorphic 
function and let 
$g(\bfw,\bar\bfw)=\vphi^*f(\bfw,\bar\bfw)=f(\vphi(\bfw,\bar \bfw))$ where $\vphi$ is a mixed homogeneous  cyclic covering associated 
with integers $a>b\ge 0$ (respectively $0\le a<b$).
Then $g$ is a holomorphic-like (resp. anti-holomorphic-like)
mixed function 
in a neighborhood of the origin.
\end{Lemma}
\begin{proof}
By an easy calculation, we get 
\[
\begin{split}
g_{w_j}=f_{ z_j}(\vphi(\bfw))
aw_j^{a-1}\bar w_j^{b},\\
g_{\bar w_j}= f_{z_j}(\vphi(\bfw))
bw_j^{a}\bar w_j^{b-1}
\end{split}
\]
and for the case $b\ge 1$  we get
\begin{multline*}
C_{j,k}=
|a\bar   w_j    w_k^{a-1} \bar   w_k^b \vphi^* f_{z_k}-a \bar   w_k   w_j^{a-1}\bar   w_j^b\vphi^*f_{z_j}|^2\\
-|b    w_j    w_k^{a} \bar   w_k^{b-1} \vphi^* f_{z_k}- b   w_k   w_j^{a}\bar   w_j^{b-1}\vphi^*f_{z_j}|^2\\
=(a^2-b^2)|  w_j  w_k|^2\left(
  w_k^{a-1}\bar   w_k^{b-1}\vphi^*f_{z_k}-  w_j^{a-1}\bar   w_j^{b-1}\vphi^*f_{z_j}
\right)^2
\end{multline*}
If $b=0$, $g$ is a holomorphic function and 
\[C_{j,k}= a^2|\bar   w_j    w_k^{a-1} \vphi^* f_{z_k}- \bar   w_k
 w_j^{a-1}\vphi^*f_{z_j}|^2\ge 0.
\]
\end{proof}
We are ready to state the first  main theorem.
\begin{Theorem}\label{holomorphic-like}
Assume  that $f(\bfz)$ is a 
convenient non-degenerate
holomorphic function.
Consider a mixed homogeneous covering 
$\vphi=\vphi_{a,b}:\BC^n\to \BC^n$, $\vphi(\bfw,\bar \bfw)=
(w_1^a\bar w_1^b,\dots, w_n^a\bar w_n^b)$ and 
let $g(\bfw,\bar\bfw)=f(\vphi(\bfw,\bar\bfw))$.
 Assume  that $a>b> 0$ (respectively $0< a< b$)
 and
consider the link $K_r:=g\inv(0)\cap {\BS}_r$.
Then
there exists a positive number $r_0$ so that 
 $g$ is strictly holomorphic-like (resp. anti-holomorphic-like)
   on $B_{r_0}$ and $K_r\subset {\BS}_r$ is a positive
 (resp. negative)  contact submanifold
 for any $r,\,0<r\le r_0$.

If $f(\bfz)$ is weighted homogeneous, 
 $g(\bfw,\bar \bfw)$ 
is strictly holomorphic-like (resp. anti-holomorphic-like)  on $\BC^n$
and $K_r\subset {\BS}_r$ is a positive contact submanifold
for any $r>0$.
\end{Theorem}
\begin{proof}
By the convenience assumption of $f(\bfz)$, 
$g(\bfw,\bar \bfw)$ is convenient.
As $f(\bfz)$ has an isolated singularity at the origin
and the restriction $\vphi:\BC^{*I}\to \BC^{*I}$ is a covering mapping for any $I\subset
 \{1,\dots,n\}$,
$g^I$ has an isolated mixed singularity at the origin.
Put $g(\bfw,\bar\bfw)=\Re g(\bfw,\bar\bfw)+i\Im g(\bfw,\bar\bfw)$. As a submanifold of $\BC^n$,
$K$ is a complete intersection variety defined by three real valued functions
$\rho= \Re g=\Im g=0$.
As the proof is completely the same, we assume that $a>b>0$.
 To prove $\al\land d\al^{n-2}$ is  positive non-vanishing on $K_r$, we can equivalently show that 
$d\rho\land \al\land d\al^{n-2}\land d\Re g\land d\Im g$ is locally non-vanishing on an arbitrary
chosen   point  $\bfz\in K_r$ and positive. 
This follows from the fact that  by the complete intersection property
$\rho, \Re g,\Im g$ can be a part of real coordinate system of $\BC^n$ near any
 point of $K_r$.
Namely there exist real analytic functions $h_4,\dots, h_{2n}$ such that 
$(\rho,\Re g,\Im g,h_4,\dots, h_{2n})$ are local coordinates.
By  Corollary \ref{contact-form} and Corollary, we have
\begin{eqnarray*}
&&d\rho\land\al\land d\al^{n-2}\land d\Re\,g\land
   d\Im\,g(\bfw,\bar\bfw)=\\
 &&i^n 2^{n-2}(n-2)!\,C(\bfw,\bar\bfw) dw_1\land d\bar w_1\land\cdots\land dw_n\land d \bar w_n\\
\end{eqnarray*}
The proof of the theorem is reduced to the following Lemma.
\end{proof}

\begin{Lemma}\label{key3} 
{\rm (1)} A smooth link $K=g\inv(0)\cap {\BS}_r$ is a contact submanifold of
 ${\BS}_r$
if and only if $C(\bfw,\bar \bfw)>0$ on $K$.

{\rm (2)}
Assume that $g=\vphi^* f$ be as in Theorem \ref{holomorphic-like} with $a>b> 0$.
Then there exists a sufficiently small neighborhood $U$ of the origin so that 
$C(\bfw,\bar\bfw)>0$ for any $\bfw\in g\inv(0)\cap U\setminus\{\bf0\}$.

If further $f(\bfz)$ is weighted homogeneous, $U$ can be the whole space
$\BC^n$.
\end{Lemma}
\begin{proof}
Recall that 
\[
C_{j,k}=(a^2-b^2)|  w_j  w_k|^2\left(
  w_k^{a-1}\bar   w_k^{b-1}\vphi^*f_{z_k}-  w_j^{a-1}\bar   w_j^{b-1}\vphi^*f_{z_j}
\right)^2
\]
Suppose that $C(\bfw,\bar\bfw)=0$ for any small neighborhood. Applying the Curve Selection Lemma
(\cite{Hamm1}), we get a real analytic curve 
$\bfw(t),\,0\le t\le 1$ such that for any $1\le j,k\le n$,
\begin{eqnarray}
\label{eq11}\,\,\,\,
\begin{cases}
& w_j w_k
\left (
w_k^{a-1}\bar w_k^{b-1} \vphi^*f_{z_k}
- w_j^{a-1}\bar w_j^{b-1}
 \vphi^*f_{z_j}
 \right )|_{\bfw=\bfw(t)}=0,\\
&g(\bfw(t),\bar\bfw(t))\equiv 0, 
\quad \bfw(t)\in \BC^{n}\setminus\{\bf 0\},\,t\ne 0.
\end{cases}
\end{eqnarray}
Let $I=\{j\,|\, w_j(t)\ne 0\}$.  
Note that $|I|\ge 2$ as each coordinate axis is not included in
 ${g}\inv(0)$ by the convenience assumption. 
By  the non-degeneracy  assumption on $f$,
there exists $j\in I$
such that $f_{z_j}(\vphi(\bfw(t),\bar \bfw(t)))\ne 0$. This implies by (\ref{eq11}),
\[
 w_k(t) f_{z_k}(\vphi(\bfw(t),\bar \bfw(t)))\ne 0
\]
for any $k\in I$. 
Take  $k\in I$ and put 
\[
 \ga(t)=w_k^{a-1}\bar   w_k^{b-1}\vphi^*f_{z_k}|_{\bfw=\bfw(t)}.
\] 
Then $\ga(t)\not\equiv 0$ and $\ga(t)$ does not depend on the choice of $k\in I$.
 Put  $ v_j(t):=\frac{dw_j(t)}{dt}$
 and 
take the differential of (\ref{eq11}). 
 As $v_j=0$ for $j\notin I$, we get

\begin{eqnarray*}
0&=&\frac{dg(\bfw(t),\bar\bfw(t))}{dt}\\
&=&\sum_{j=1}^n f_{z_j}(\vphi(\bfw(t),\bar\bfw(t)))(a w_j(t)^{a-1}\bar w_j(t)^{b} {v}_j(t)+b w_j(t)^a \bar w_j(t)^{b-1} \bar v_j(t))\\
&=&\ga(t)\sum_{j=1}^n (a \,\bar w_j(t) v_j(t)+b\, w_j(t) \bar  v_j(t))\\
&=&\ga(t) \frac{(a+b)}2 \frac{d\|\bfw(t)\|^2}{dt}.
\end{eqnarray*}
Thus  $\frac{d\|\bfw(t)\|^2}{dt}\equiv 0$ .
The last equality is derived from 
\[
\frac{d(\|\bfw(t)\|^2}{dt}=\sum_{j=1}^n (w_j(t)\bar v_j(t)+
\bar w_j(t) v_j(t))=2\Re \sum_{j=1}^n w_j(t)\bar{v}_j(t).
\]
This  implies that $\|\bfw(t)\|$ is constant
  which is a contradiction to the assumption $\|\bfw(t)\|\to 0\,(t\to 0)$.

%
We prove the second assertion.
Assume that $f(\bfz)$ is a weighted homogeneous polynomial
of degree $d$ with a weight vector
 $P=(p_1,\dots, p_n)$ and $a>b > 1$.
Then $g=\vphi^*f$ is a strongly polar weighted homogeneous 
 polynomial
with
$\rdeg_P g=(a+b)d$ and $\pdeg_Pg=(a-b)d$.
Put $I=\{j\,|\, w_j\ne 0\}$.
Assume that $C_{j,k}(\bfw,\bar\bfw)=0$ for any  $j,k$  for some $\bfw\in g\inv(0)\setminus \{0\}$.
Put 
\[
 \ga=w_k^{a-1}\bar w_k^{b-1}\vphi^*f_{z_k} 
\]
for any fixed $ k\in I$.
Note that $\ga$ is independent of $k\in I$.
As $f^I$  (and also $g^I$ ) has an isolated singularity at the origin, $\ga\ne 0$.
Then this implies that 
\[
g_{w_j}=aw_j^{a-1}\bar w_j^{b}\vphi^*f_{z_j}=
a\ga\bar w_j,\,j\in I
\]
and by Euler equality (\ref{strongEuler}), we get 
a contradiction
\[\begin{split}
 0=a\, d\, g(\bfw,\bar\bfw)&=\sum_{j=1}^n p_j w_j  g_{w_j}\\
&=\sum_{j\in I} p_j w_j  g_{w_j}\\
&=\ga a\,\sum_{j\in I} p_j|w_j|^2\ne 0. 
\end{split}
\]

\end{proof}

\begin{Corollary}\label{holom}
Assume  that $f(\bfz)$ is a
holomorphic function with isolated singularity at the origin.
Consider the link $K_r:=g\inv(0)\cap {\BS}_r$.
Then
there exists a positive number $r_0$ so that 
$K_r\subset {\BS}_r$ is a positive contact submanifold
 for any $0<r\le r_0$.

If $f(\bfz)$ is weighted homogeneous, 
 $K_r\subset {\BS}_r$ is a positive contact submanifold
for any $r>0$.
\end{Corollary}
\begin{proof}The proof is parallel to that of Lemma \ref{key3}.
Recall that 
\[C_{j,k}=|a\bar   z_j    f_{z_k}-a \bar   z_k f_{z_j}|^2.
\]
We do the same argument.  If $\{\bfz\in \BC^n\,|\, C(\bfz,\bar\bfz)=0\}\cap f\inv(0)$ is not isolated at the origin,
we  take an analytic curve $\bfz(t)\in\{\bfz\,|\,C(\bfz,\bar\bfz)=0\}\cap f\inv(0)$ as above.
Putting $I=\{j\,|\, z_j(t)\ne 0\}$  as above, we get  $|I|\ge 2$.
As $f|{\BC^I}$ has an isolated singularity, we
can assume that $f_{z_k}(\bfz(t))\ne 0$ for some $k\in I$.
Take $j\in I$ with $j\ne k$. Then $C_{j,k}=0$
implies that $\bar z_j(t) f_{z_k}(\bfz(t))\ne 0$ and thus 
  $k\in I$ and $f_{z_j}(\bfz(t))\ne 0$.
Put $c(\bfz)=f_{z_k}(\bfz)/\bar z_k$ for a fixed  $k\in I$. 
This is a non-zero and independent of $k\in I$.
Note that $v_k(t)=0$ for $k\notin I$. Therefore
 we get 
 \begin{eqnarray*}
 0&=&\frac{df(\bfz(t))}{dt}=\sum_{j=1}^n f_{z_j}(\bfz(t))v_j(t)\\
&=&
\sum_{j\in I} f_{z_j}(\bfz(t))v_j(t)=
 c(\bfz(t))\sum_{j\in I} \bar z_j(t)v_j(t)
 \end{eqnarray*}
and we get the same contradiction $\frac{d\|\bfz(t)\|^2}{dt}\equiv 0$.

Finally assume further $f(\bfz)$ is weighted homogeneous of degree $d$
with weight vector $P=(p_1,\dots, p_n)$.
Put $I=\{j\,|\, z_j\ne 0\}$ as above and put $c(\bfz)=f_{z_k}(\bfz)/\bar z_k$ for $k\in I$. 
Then $f_{z_j}(\bfz)=c(\bfz) \bar z_k$ and  we get the same contradiction:
\begin{eqnarray*}
0=f(\bfz)=\sum_{j=1}^n p_j z_jf_{z_j}(\bfz)=\sum_{j=1}^n c(\bfz) p_j|z_j|^2\ne 0.
\end{eqnarray*}
\begin{Remark}
Corollary  \ref{holom} gives 
a simple proof of holomorphic link 
to be a contact submanifold without using the strict  pseudo-convex property.  
\end{Remark}

\end{proof}
\section{Open book structure.}
\subsection{Open book}
{\em An open book with binding $N$} on an oriented manifold $M$ 
of dimension $2n-1$ is a couple $(N,\theta)$ where
$N$ is a codimension two submanifold with a trivial normal bundle and 
$\theta: M\setminus N\to S^1$ is a local trivial smooth fibration where
$\theta$ coincide with the angular coordinate of the trivial tubular neighborhood $ N\times D_\de\subset M$ (\cite{Giroux, C-N-PP}).
The orientation of $M$ gives a canonical orientation to the fiber $F_{\eta}:=\theta\inv(\eta)$
for each $\eta\in S^1$. Restricting the fibration on $M\setminus N\times D_\de$, 
the fiber $F_\eta':=F_\eta\cap (M\setminus N\times D_\de)$ is a manifold
with boundary $N$. Thus $N$ has also a canonical orientation.
\subsection{Contact structure carried by an open book}
Assume that we have a contact form $\xi$ defined by a global 1-form $\al$ as before.
We say that a contact structure $\xi$  is
{\em  carried by an open book }
 $(N,\theta)$ if the following  are satisfied (\cite{Giroux}).
 \begin{enumerate}
  \item  The restriction of $\al$ to $N$ is a contact form on $N$.
 \item The two-form  $d\al$ defines a symplectic form of each fiber $F_\eta=\theta\inv(\eta)$.
 \item The orientation of $N$ induced by $\al$ is the same as that of the boundary of $F_\eta$.
 \end{enumerate}
Recall that the condition (2)
is equivalent to 
$d\theta(R)>0$
where $R$ is the Reeb vector field (\cite{Giroux,Geiges}).
 For further detail about  a contact structure carried with an open book
   and symplectic structures, see  
   H. Geiges \cite{Geiges}, 
 Giroux \cite{Giroux}, R. Berndt \cite{Berndt} and Caubel-Nemethi-Popescu-Pampu \cite{C-N-PP}.
 
\subsection{Milnor open book for mixed functions.}
Let $g(\bfz,\bar\bfz)$ be a convenient strongly 
non-degenerate  mixed function.
Let $V=g\inv(0)$ and we assume that $V$ has an isolated mixed singularity at the origin.
By Theorem 33 (\cite{OkaMix}), we have 
\begin{Theorem}\label{open-book}
For a sufficiently small $r$,  the mapping 
\begin{eqnarray} \label{Milnor-fibration}
g/|g|:\quad {\BS}_r\setminus K_r\to S^1
\end{eqnarray}
is a locally trivial fibration.
\end{Theorem}
By the transversality, we have a trivial tubular neighborhood
$K_r\times D_\de$ such that the following diagram commutes.
\[
\begin{matrix}
K_r\times D_\de&\subset &K_r\times  D_\de^*&\subset &{\BS}_r\setminus K_r\\
&&\mapdown{p}&&\mapdown{g/|g|}\\
&&D_\de^*&\mapright{normal}&S^1
\end{matrix}
\]
where $D_\de^*=\{\eta\in \BC\,|\,0\ne \eta, |\eta|\le \de\}$,
$p$ is the  second projection and {\em normal}  is the normalization
map $\eta\mapsto \eta/|\eta\|$.
The argument $\theta$  is characterized by the equality:
\[\log g(\bfz,\bar\bfz)=\log|g(\bfz,\bar\bfz)|+i\theta.
\]
From this and the obvious equality $|g(\bfz,\bar\bfz)|^2=g(\bfz,\bar\bfz)\bar g(\bfz,\bar\bfz)$,
we have
\begin{Proposition}
\begin{eqnarray}
&\nabla \theta=i \left(
\cfrac{\overline{ g_{z_1}}} {\bar g}-\cfrac  {g_{\bar z_1}}g,
\dots, \cfrac {\overline{g_{z_n}}}{\bar g}-\cfrac  {g_{\bar z_n}}g 
\right),
\end{eqnarray}
\begin{eqnarray}
d\theta=-i\left(
\frac{\partial g+\bar\partial g}{g}-\frac{\partial \bar g+\bar \partial \bar g}{\bar g}
\right ),
\end{eqnarray}
where $\partial,\bar\partial$  are defined  for a mixed function $h$ by 
\[dh=\partial h+\bar \partial h,\quad \partial h=\sum_{j=1}^n h_{z_j}dz_j,\quad
\bar\partial h=\sum_{j=1}^nh_{\bar z_j}d\bar z_j.
\]
\end{Proposition}
\subsection{Contact structure carried by a Milnor open book.}
We consider the existence problem of the contact structure carried by  a Milnor open book for
polar weighted homogeneous mixed functions. 
We consider a homogeneous mixed covering lifting
$g(\bfw,\bar\bfw)=f(w_1^a\bar w_1^b,\dots, w_n^a\bar w_n^b)$ with
$a>b\ge 0$
where $f(\bfz)$ is a convenient non-degenerate holomorphic
function defined in a neighborhood of the origin.
\subsubsection{Strongly polar homogeneous case.}
First we consider the easy case that $f(\bfz)$ is a 
 homogeneous polynomial of degree $d$.
Then $g(\bfw,\bar\bfw)$ is a strongly polar homogeneous polynomial
with $\rdeg\, g=d(a+b),\, \pdeg\,g=d(a-b)$.
In this case, we assert:
\begin{Theorem} Assume that $f(\bfz)$ is a  homogeneous polynomial
with $\pdeg\,f=d > 0$.
The canonical contact form $\al$ and $\omega$ is adapted with 
the Milnor open book
$\theta: S_r^{2n-1}\setminus K_r\to S^1$ for any $r>0$.
\end{Theorem}
\begin{proof}
Recall that 
Reeb vector field $R$ and $d\theta$ are given  on ${\BS}_r$ by 
\[\begin{split}
 R(\bfz)&=\frac i{2 r^2}\sum_{j=1}^n (z_j\frac{\partial}{\partial z_j}-
\bar z_j\frac{\partial}{\partial\bar z_j})\\
d\theta&=-i\left\{
\frac{\partial g+\bar\partial g}{g}-\frac{\partial \bar g+\bar \partial \bar g}{\bar g}
\right\}
\end{split}
\]
In fact, we use the polar Euler equality:
\[
\sum_{j=1}^n\left (
 z_j g_{z_j}-\bar z_j g_{\bar z_j}\right)=d(a-b) g
 \]
 and its conjugate:
 \[\sum_{j=1}^n\left (
 \bar z_j {\bar g}_{\bar z_j}-z_j {\bar g}_{z_j}\right)=d(a-b)\bar  g.
 \]
 Using these Euler equalities, we get:
 \[
 \begin{split}
 d\theta(R)&=\frac{\partial g(R)+\bar\partial g (R)}g-\frac{\partial \bar g(R)+\bar\partial \bar g(R)}{\bar g}\\
 &=\sum_{j=1}^n \frac{z_j g_{z_j}-\bar z_jg_{\bar z_j}}{g}
 -\sum_{j=1}^n \frac{z_j {\bar  g}_{z_j}-\bar z_j {\bar g}_{\bar z_j}}{\bar g}\\
 &=2\,d\,(a-b)>0.
 \end{split}
 \]
This shows that for any radius $r>0$, the canonical contact structure on $K_r\subset {\BS}_r$ is adapted
 with  the Milnor  fibration.
\end{proof}
\section{General case}
We are interested in the 
existence of open book structure adapted to the contact structure
which is the restriction of $\al$ to the link $K_r\subset {\BS}_r$ where 
$K_r=g\inv(0)\cap {\BS}_r$. We have shown  that there exists a canonical Milnor
 fibration
on $g/|g|:{\BS}_r\setminus\{K_r\}\to S^1$ by \cite{OkaMix}.
However this fibration is not adapted with the symplectic structure
given by $d\al$. Therefore
we will change the contact form $\al$ without changing the contact
structure $\xi$ so that the new contact form will be carried by the 
Milnor open book.
We follow the proof  of Theorem 3.9 in \cite{C-N-PP}  for the holomorphic functions in 
Caubel-N\'emethi-Popescu-Pampu.

We modify the contact form $\al$ by
\[
 \al_c=e^{-c|g|^2}\al
\]
 with a sufficiently large positive real number  $c>0$.
This does not change the contact structure $\xi=\Ker\,\al$ but the two form $\omega_c=d\al_c$ is changed
as 
\[
 \omega_c=d(e^{-c|g|^2})\land \al+e^{-c|g|^2} d\al
\]
and the corresponding symplectic structure changes.
Consider  the new Reeb vector field $R_c$. Put $H=e^{-c|g|^2}$.
Put also $R_c=k(R+S_c)$ with $S_c$ is tangent to $\xi$.
Then we get $kH=1$. As
\[
 d\al_c=dH\land \al+H\,d\al,
\]
the condition for $R_c$ to be the Reeb vector field
 $\iota_{R_c}d\al_c|_{\xi}=0$
gives the condition:
\begin{eqnarray}\label{eq16}
\iota_{S_c}\omega|_{\xi}=\frac{dH}H|_{\xi}=-c\,d|g|^2.
\end{eqnarray}
Put $\pi: T_{\bfw}\BC^n\to \xi(\bfw)$ be the hermitian  orthogonal projection.
Namely, $\pi(\bfv)=\bfv-(\bfv, \tilde R) \tilde R$ and 
$\tilde R=R/\|R\|$.
Then (\ref{eq16}) implies  by (\ref{omega-formula}) that 
\begin{eqnarray*}
\iota_{S_c}\omega|_{\xi}&=& -c\,d|g|^2|_{\xi}=\Re(-c\gtr |g|^2,\cdot)|_{\xi}\\
&=&\frac 14\omega(i c \gtr|g|^2,\cdot)|_{\xi}
=\omega(\pi(i c\gtr|g|^2)/4,\cdot)|_{\xi}\\
&=&\iota_{\pi(i c\gtr|g|^2)/4}\omega|_{\xi}.
\end{eqnarray*}
As $\omega$ is non-degenerate on $\xi$,
we get 
\begin{eqnarray}\label{different}
S_c=\pi(ic \gtr|g|^2/4)
\end{eqnarray}
Thus we get 
\begin{eqnarray}
 |g|^2d\theta(R_c) &=& k|g|^2d\theta(R)+k\Re(|g|^2\gtr\theta,S_c)\\
&=& k|g|^2d\theta(R)+k\Re(\pi(|g|^2\gtr\theta),\pi(ic \gtr|g|^2/4)).
\end{eqnarray}
For simplicity, we introduce two vectors
\begin{eqnarray}
\gtr_{\partial} g&=(g_{w_1},\dots, g_{w_n})\\
\gtr_{\bar\partial} g&=(g_{\bar w_1},\dots, g_{\bar w_n}).
\end{eqnarray}
\begin{Remark}
In  our previous paper \cite{OkaPolar}, we used the notation $dg$ and $\bar dg$
instead of  $\gtr_{\partial}g$ and $\gtr_{\bar\partial}g$. We changed notations as the previous
notations are confusing with 1-forms $\partial g,\,\bar\partial g$.
We use $dg$ not for $\partial g$ but $dg=(\partial+\bar \partial)g$.
\end{Remark}
Recall that 
\begin{eqnarray}
\gtr |g|^2&=2g\,\overline{\gtr_{\partial} g}+2 \bar g \gtr_{\bar\partial} g\\
|g|^2\gtr \theta &=ig \overline{\gtr_{\partial} g}-i\bar g \gtr_{\bar\partial} g.
\end{eqnarray}
Thus $2|g|^2\gtr\theta$ and $i\gtr |g|^2$ are different in the case of  mixed functions.
This makes a difficulty. (In the holomorphic function case, 
they are the
same up to a scalar multiplication, as $\gtr_{\bar\partial} g\,  $ vanishes.)

Put $\pi':\BC^n\to \BC\cdot R$ be the orthogonal projection
to the complex line $\BC\cdot R=\BC\cdot \bfw$ generated by $R$ or $\bfw$. Namely
$\pi'(\bfv)=(\bfv,\bfw)\bfw/\|\bfw\|^2=(\bfv,R)R/\|R\|^2$.
Then $\pi(\bfv)=\bfv-\pi'(\bfv)$. Consider the expression:
\begin{eqnarray}
&g\,\overline{\gtr_{\partial} g}=\bfv_{11}+\bfv_{12},\,\quad
\begin{cases} &\bfv_{11}=\pi(g\,\overline{\gtr_{\partial} g})\\
              &\bfv_{12}=\pi'(g\,\overline{\gtr_{\partial} g})
\end{cases}\\
&\bar g \gtr_{\bar\partial} g=\bfv_{21}+\bfv_{22},\,\quad
\begin{cases} &\bfv_{21}=\pi(\bar g\,{\gtr_{\bar\partial} g})\\
              &\bfv_{12}=\pi'(\bar g\,{\gtr_{\bar\partial} g}).
\end{cases}
\end{eqnarray}
Using this expression, we get
\begin{eqnarray}
\pi(ic\gtr|g|^2/4)&=\cfrac{i c}{2} (\bfv_{11}+\bfv_{21})\\
\pi(|g|^2\gtr \theta)&=i(\bfv_{11}-\bfv_{21})
\end{eqnarray}
Thus we get
\begin{eqnarray}
 |g|^2d\theta(R_c)
&= & k|g|^2d\theta(R)+k\Re(\pi(|g|^2\gtr\theta),\pi(ic \gtr|g|^2/2))\\
&=&k|g|^2d\theta(R)+ \frac{ck}2 (\|v_{11}\|^2-\|v_{21}\|^2).\notag
\end{eqnarray}
Here we have used the equality: $\Re(\bfv_{21},\bfv_{11})= \Re(\bfv_{11},\bfv_{21})$.
The key assertion is the following.
\begin{Lemma}\label{positivity}
We have the inequality: $\|v_{11}\|^2-\|v_{21}\|^2\ge 0$
and the equality takes place if and only if
$\overline{\gtr_{\partial} g}(\bfw,\bar\bfw)=\la_1\bfw$ and
 $\gtr_{\bar\partial}g(\bfw,\bar\bfw)=\la_2\bfw$
for some $\la_1,\la_2\in\BC$.
In this case, we have also
$\gtr\theta=\la R$ for some $\la\in \BC$.
\end{Lemma}
\begin{proof}
Let $\bfv_1=g\overline{\gtr_{\partial} g}$ and
 $\bfv_2=\bar g \gtr_{\bar\partial}g$.
As $\{\bfv_{11},\bfv_{12}\}$ and $\{\bfv_{12},\bfv_{22}\}$ are hermitian
 orthogonal, we have
\begin{eqnarray*}
 \|\bfv_{11}\|^2&=&\|\bfv_1\|^2-\|\bfv_{12}\|^2,\quad
\|\bfv_{21}\|^2=\|\bfv_2\|^2-\|\bfv_{22}\|^2.
\end{eqnarray*}
We go now further precise expression. Put
\[
 \bfv_1=(v_1^1,\dots, v_1^n),\quad \bfv_2=(v_2^1,\dots, v_2^n).
\]
Then we have 
\begin{eqnarray*}
v_{1}^j&=&g(\bfw,\bar\bfw)\, \,\overline{f_{z_j}(\vphi(\bfw,\bar\bfw))}\,\,
a\bar w_j^{a-1}w_j^b\\
v_{2}^j&=&\bar g(\bfw,\bar\bfw)\,\, f_{ z_j}(\vphi(\bfw,\bar\bfw))\,\,
bw_j^a \bar w_j^{b-1}.
\end{eqnarray*}
Recall that $R=i\bfw/2\rho(\bfw)$.
Thus
\[\begin{split}
 \bfv_{12}&=\left(\sum_{j=1}^n g(\bfw,\bar\bfw)\,\overline{f_{z_j}(\vphi(\bfw,\bar\bfw))}
a\bar w_j^{a}w_j^b\right)\bfw/\|\bfw\|^2\\
&=
a\,g(\bfw,\bar\bfw)\left(\sum_{j=1}^n\overline{f_{z_j}}(\vphi(\bfw,\bar\bfw))\bar w_j^{a}w_j^b\right)
\bfw/\|\bfw\|^2\\
\bfv_{22}&=\left(\sum_{j=1}^n \bar g(\bfw,\bar\bfw) f_{z_j}(\vphi(\bfw,\bar\bfw)) b\,w_j^a \bar w_j^b\right)
   \bfw/\|\bfw\|^2\\
&=b\bar (\bfw,\bar\bfw)g(\bfw,\bar\bfw)\left(\sum_{j=1}^n  f_{z_j}(\vphi(\bfw,\bar\bfw)) w_j^a \bar w_j^b\right)
   \bfw/\|\bfw\|^2.
\end{split}
\]
Thus  
we get 
\begin{eqnarray*}
0\le \|\bfv_{11}\|^2&=&\|\bfv_1\|^2-\|\bfv_{12}\|^2\\
&=&a^2|g|^2\sum_{j=1}^n |f_{z_j}|^2|w_j|^{2(a+b-1)}-a^2|g|^2
\left|\sum_{j=1}^n\overline{f_{z_j}}\bar w_j^{a}w_j^b\right|^2\\
&=&a^2|g|^2(\ga-\be)\\
0\le \|\bfv_{21}\|^2&=&\|\bfv_2\|^2-\|\bfv_{22}\|^2\\
&=& b^2|g|^2\sum_{j=1}^n |f_{z_j}|^2|w_j|^{2(a+b-1)}-b^2|g|^2
\left|\sum_{j=1}^n{f_{z_j}} w_j^{a}\bar w_j^b\right|^2\\
&=&b^2|g|^2(\ga-\be)\\
&&\text{where}\,\,
\begin{cases}
\ga&=\sum_{j=1}^n |f_{z_j}|^2|w_j|^{2(a+b-1)}\\
\be&=\left|\sum_{j=1}^n\overline{f_{z_j}}\bar w_j^{a}w_j^b\right|^2.
\end{cases}
\end{eqnarray*}
Thus  $\ga\ge \be$ and we have 
\[\begin{split}
&\|\bfv_{11}\|^2-\|\bfv_{21}\|^2=(a^2-b^2)|g|^2(\ga-\be)\ge 0
\end{split}
\]
and the equality holds if and only if $\ga=\be$. This is equivalent to 
$\|\bfv_{11}\|=\|\bfv_{21}\|=0$ and this  implies
$\overline{\gtr_{\partial} g}(\bfw,\bar\bfw)=\la_1\bfw$ and
 $\gtr_{\bar\partial}g(\bfw,\bar\bfw)=\la_2\bfw$
for some $\la_1,\la_2\in\BC$.
The last assertion follows from
$|g|^2\gtr\theta=i(\bfv_1-\bfv_2)$.

\end{proof}
\subsection{Main theorem}
Now we are ready to  state our main theorem.
Let $\vphi(\bfw,\bar\bfw)=(w_1^a\bar w_1^b,\dots, w_n^a\bar w_n^b)$ with
$a>b> 0$ as before.
\begin{MainTheorem}\label{MainTheorem2}
Assume that $f(\bfz) $ is a convenient  
non-degenerate 
holomorphic
 function
so that $g(\bfw,\bar\bfw)=\vphi^*f(\bfw,\bar\bfw)$ is a convenient 
non-degenerate
 mixed function
 of strongly polar 
weighted homogeneous face type. 
Then there exists a positive number $r_0$ such that the Milnor open book
$f/|f|:\, {\BS}_r\setminus K_r\to S^1$ carries a contact structure for any
 $r>0$ with $r\le r_0$.

If further $f(\bfz)$ is a 
 weighted homogeneous
 polynomial,
we can take  $r_0=\infty$ and any $r>0$.
\end{MainTheorem}
\begin{proof}

For the proof, we do the same discussion as that of Caubel-N\'emethi-Popescu-Pampu
\cite{C-N-PP}.
Let 
\[
Z_\delta:=\{\bfw\in S_r^{2n-1}\setminus V_\de\,|\, d\theta(R)\le 0\},
\quad
 V_\de=S_r^{2n-1}\cap g\inv (D_\delta)
\]
where $\de$ is sufficiently small so that $f\inv(0)$ and ${\BS}_r$ are
transverse and  $V_\de$ is a trivial tubular neighborhood.
Let $\al_c$ and $R_c$ be as before.
As we have shown that 
 \begin{eqnarray*}
 |g|^2d\theta(R_c)
=k|g|^2d\theta(R)+ \frac{ck}2 (\|v_{11}\|^2-\|v_{21}\|^2)\notag
\end{eqnarray*}
with $k=1/e^{-c|g|^2}$ and the second term is non-negative  and the 
 equality holds ({\em i.e.} $ \|v_{11}\|^2-\|v_{21}\|^2=0$) if and only if 
$\nabla \theta=\la R$.   In this case,  $d\theta(R)=\Re \la\|R\|^2$ and $d\theta(R)$
is positive if  $\Re \la>0$. Thus taking sufficiently
 large $c>0$, we only need to show that $\nabla \theta$ and $R$ are
 linearly independent on $Z_\de$. Thus 
the following lemma
 completes the proof.
 Compare with Proposition 3.8 (\cite{C-N-PP}).
\end{proof}
\begin{Lemma}\label{Positive1}
Assume that $\nabla \theta(\bfw)=\la R(\bfw)$ on $\bfw\in Z_{\delta}$
for some $\la\in \BC$.
\begin{enumerate}
\item Assume that $f(\bfz)$ is convenient non-degenerate weighted
      homogeneous polynomial.
Then $\la$ is a positive real number. 
\item If $f(\bfz)$ is not weighted homogeneous but a convenient non-degenerate
mixed function of strongly polar weighted homogeneous face type,
there exists a positive number $r_0$  such that 
 $ \Re\la$ is positive for any $\bfw\in {\BS}_r\setminus g\inv(0)$ and 
$r\le r_0$. 
\end{enumerate}
\end{Lemma}
\begin{proof}
{\bf (1)} Assume that  $f(\bfz)$ is a weighted homogeneous of degree
 $d$ with weight vector
 $P=(p_1,\dots, p_n)$.
Let $m_r=(a+b)d$ and $m_p=(a-b)d$, the radial and polar degree of $g(\bfw,\bar\bfw)$.
The assumption says that
\begin{eqnarray}\label{str-pos}
\la\,w_j=\frac 1{\bar g(\bfw,\bar\bfw)}
\overline{g_{w_j}}(\bfw,\bar\bfw)-\frac 1{g(\bfw,\bar\bfw)}
g_{\bar w_j}(\bfw,\bar \bfw)
\end{eqnarray} for $j=1,\dots, n$.
Taking  the summation of (  \ref{str-pos})$\times  p_j\bar w_j $ for $j=1,\dots,n$, we get:
\[\begin{split}
\la\sum_{j=1}^n p_j|w_j|^2&=\frac 1{\bar g(\bfw,\bar\bfw)}\sum_{j=1}^n
   p_j\bar w_j\overline{g_{w_j}}(\bfw,\bar\bfw)-\frac 1{
   g(\bfw,\bar\bfw)}\sum_{j=1}^n p_j\bar w_jg_{\bar w_j}(\bfw,\bar\bfw)\\
%
&=(m_r+m_p)-(m_r-m_p)=2m_p>0
\end{split}
\]
by the  strong Euler equalities (  \ref{strongEuler}).
This implies $\la$ is a positive number.

{\bf (2) General case.} Assume that $g(\bfw,\bar \bfw)$ is a convenient
 non-degenerate mixed function  of strongly polar weighted homogeneous
 face type.
Assume that the assertion (2) does not hold. Using Curve Selection Lemma (\cite{Milnor, Hamm1}), we can find a real analytic curve
$\bfw(t)\in \BC^n\setminus f\inv(0)$ for $0<t\le \eps$ and 
Laurent  series $\la(t)$ such that 
\begin{eqnarray}\label{eq3}
\bigtriangledown \theta(\bfw(t))=\la(t) R(\bfw(t))
\end{eqnarray}
such that $\nabla \theta (R(\la(t)))\le 0$.
We show this give a   contradiction by showing  $\lim_{t\to 0}\arg
 \la(t)=0$.

Let $I=\{j\,|\, w_j(t) \not\equiv 0\}$. Then  $|I|\ge 2$ and we restrict our the discussion to
the coordinate subspace $\BC^I$ and $g^I=g|_{\BC^I}$.
For the notation's simplicity, we assume that $I=\{1,\dots, n\}$ hereafter.
Consider the Taylor (Laurent) expansions:
\begin{eqnarray*}
w_j(t)= a_j t^{p_j}+\text{(higher terms)},\quad j=1,\dots, n,\\
\la(t)=\la_0 t^{\ell}+\text{(higher terms)},\\
g(\bfw(t),\bar\bfw(t))=g_0 t^d+\text{(higher terms)}
\end{eqnarray*}
where $a_j, \la_0, g_0\ne 0$ and $p_j\in \BN,\, \ell\in \BZ$.
Then the equality (  \ref{eq3}) says
\begin{eqnarray}\label{16}
\la(t)\,w_j(t)= 
\cfrac{\overline{g_{w_j}}(\bfw(t),\bar\bfw(t))}{{\bar g(\bfw(t),\bar\bfw(t))}}
-  
\cfrac{g_{\bar w_j}(\bfw(t),\bar \bfw (t))}{{g(\bfw(t),\bar\bfw(t))}} ,
%
\end{eqnarray}
for $j=1,\dots,n$.
Consider the weight vector $P =(p_1,\dots, p_n)$ and the  face function
 $f_P$.
Then $g_P=\vphi^* f_P$.
By (  \ref{16})  we get the equalities:
\begin{eqnarray}\label{17}
\la_0a_j t^{p_j+\ell}+\dots=
\left(\cfrac{
\overline{(g_P) w_j}(\bfa)}{\bar {g_0}} 
-\cfrac{(g_P)_{ \bar w_j}(\bfa)}{g_0}\right) t^{d(P;f)-p_j-d}
+\dots.
\end{eqnarray}
where  $j=1,\dots, n$ and $\bfa=(a_1,\dots, a_n)$.
The order of the left side for $j$ is
$p_j+\ell$. The order of the right side is 
at least $d(P;f)-p_j-d$. 
Thus we have 
\begin{eqnarray}\label{17bis}
p_j+\ell\ge d(P;f)-p_j-d.
\end{eqnarray}
 Put
\[
C_j:=\left(
\cfrac{\overline{(g_P)_{w_j} }(\bfa)}{\bar {g_0}}- \cfrac{(g_P)_{\bar w_j}(\bfa)}{g_0}
\right).
\]
The equality in (  \ref{17bis}) holds if $C_j\ne 0$:
\begin{eqnarray}\label{18}
p_j+\ell=d(P,g)-p_j-d,\quad \text{if}\,\,C_j\ne 0.
\end{eqnarray}
We assert that 
\begin{Assertion}  There exists some $j$ such that $C_j\ne 0$.
\end{Assertion}
\begin{proof}
Assume that $C_1=\dots=C_n=0$. This implies that
$\overline{\nabla_{\partial}g(\bfa)}=
u\nabla_{\bar\partial}g(\bfa)$ with $u=\bar b/b$
and therefore we see  that
 $\bfa$ is a critical point of
$g_P:\BC^{*n}\to \BC$ by Proposition 1 (\cite{OkaPolar})
 which contradicts to the non-degeneracy assumption.
\end{proof}
Let  $p_{min}=\min\,\{p_j\,|\,\,j=1,\dots,n\}$ and $J=\{j\,|\, p_j=p_{min}\}$
and let $p_{max}=\max\{p_j\,|\,  C_j\ne 0\}$ and $J'=\{j\,|\, p_j=p_{max},\,C_j\ne 0\}$.
We assert that 
\begin{Assertion} $p_{min}=p_{max}$.
\end{Assertion}
\begin{proof}
Assume that $p_{min}<p_{max}$.
Then  we have a contradiction: For $k\in J$ and $j\in J'$,
\[
p_k+\ell<p_{j}+\ell=d(P,g)-p_j-d<d(P,g)-p_k-d
\]
which contradicts to (  \ref{18}).
\end{proof}
Thus we have proved the equivalence $C_j=0\iff j\notin J$ and 
comparing the leading  coefficients of
(\ref{17}),
\begin{eqnarray}
&\la_0\,a_k=C_k,\,\, p_{min}+\ell=d(P,g)-p_{min}-d,\,\forall k\in J. \label{eq1}
\end{eqnarray}
Then taking  the summation $\sum_{j\in J} p_j\bar a_j\times (  \ref{eq1})$ , we get the equality
\begin{eqnarray}
\sum_{k\in J} p_k\la_0\bar a_k  a_k=\sum_{k\in J}p_k\bar a_k C_k.
\end{eqnarray}
The left side is $\la_0\sum_{k\in J} p_k \,|a_k|^2\ne 0$.
The right side is
\begin{eqnarray*}
\sum_{k\in J}p_k\bar a_k C_k&=&\sum_{k=1}^np_k\bar a_k C_k\\
&=&\sum_{k=1}^np_k\bar a_k\left(
\overline{(g_P)_{w_j}}(\bfa,\bar\bfa)/\bar {g_0}
-  (g_P)_{\bar w_j}(\bfa,\bar\bfa)/ g_0
\right)\\
&=&(\rdeg(P,g_P)+\pdeg(P,g_P))\overline{ g_P}(\bfa,\bar \bfa)/{\bar {g_0}}\\
&&-(\rdeg(P,g_P)-\pdeg(P,g_P))g_P(\bfa,\bar \bfa)/{g_0}.
\end{eqnarray*}
As the left side is non-zero,  we have $g_P(\bfa,\bar \bfa)\ne 0$ and 
$d=d(P,g)$ and $g_0=g_P(\bfa,\bar\bfa)$.
Thus we finally obtain the equality
\[
\la_0\sum_{k\in J}p_k|a_k|^2=2\pdeg(P,g_P)
\]
which implies $\la_0>0$ and thus $\lim_{t\to 0}\arg\,\la(t)=0$.
As $d\theta(R(\bfw(t)))=\Re\nabla\theta(R(\bfw(t)))=\Re
 \la(t)\|R(\bfw(t))\|^2>0$ for a sufficiently small $t$, this is a contradiction.
\end{proof}
\begin{Remark}
1. Lemma \ref{Positive1} hold for any convenient non-degenerate mixed
 function 
$g(\bfw,\bar \bfw)$ of strongly polar weighted homogeneous face type,
as the proof do not use the assumption $g=\vphi^*f$.

2. Let $g=\vphi^*f$ where $f(\bfz)$ is a convenient non-degenerate
 holomorphic function and $\vphi_{a,b}$ is a homogeneous cyclic covering
 map.
Then 
\end{Remark}
\begin{Assertion} The link topology of $g\inv(0)$
is a combinatorial invariant and it is determined by 
$\Ga(f)$.
\end{Assertion}
\begin{proof}
Assume that $f'(\bfz)$ is another convenient non-degenerate holomorphic
 function with $\Ga(f')=\Ga(f)$ and $let g'=\vphi^* f'$. Take a one-parameter family
 $f_t(\bfz),\,0\le t\le 1$
so that $\Ga(f_t)=\Ga(f)$, $f_0=f, f_1=f'$ and $f_t(\bfz)$
is non-degenerate for any $t$. The we get a one-parameter family
$g_t:=\vphi^*f_t$ of mixed function of strongly polar weighted homogeneous
 face type. Then their links are certainly isotopic.
\end{proof}
Observe that there exists in general mixed functions $h(\bfw,\bar\bfw)$
which is convenient, non-degenerate and of strongly polar weighted
 homogeneous face type but it is not a homogenous lift of a holomorphic function.
In such a case, the topology of the links of $g$ and $h$ may be
 different.
See Example 5.4 in \cite{OkaVar}.

\def\cprime{$'$} \def\cprime{$'$} \def\cprime{$'$} \def\cprime{$'$}
  \def\cprime{$'$} \def\cprime{$'$} \def\cprime{$'$} \def\cprime{$'$}


\begin{thebibliography}{1}

\bibitem{Berndt}
R.~Berndt.
\newblock {\em An introduction to symplectic geometry}, volume~26 of {\em
  Graduate Studies in Mathematics}.
\newblock American Mathematical Society, Providence, RI, 2001.
\newblock Translated from the 1998 German original by Michael Klucznik.

\bibitem{C-N-PP}
C.~Caubel, A.~N{\'e}methi, and P.~Popescu-Pampu.
\newblock Milnor open books and {M}ilnor fillable contact 3-manifolds.
\newblock {\em Topology}, 45(3):673--689, 2006.

\bibitem{Geiges}
H.~Geiges.
\newblock {\em An introduction to contact topology}, volume 109 of {\em
  Cambridge Studies in Advanced Mathematics}.
\newblock Cambridge University Press, Cambridge, 2008.

\bibitem{Giroux}
E.~Giroux.
\newblock Contact structures and symplectic fibrations over the circle.

\bibitem{Hamm1}
H.~Hamm.
\newblock Lokale topologische {E}igenschaften komplexer {R}\"aume.
\newblock {\em Math. Ann.}, 191:235--252, 1971.

\bibitem{Milnor}
J.~Milnor.
\newblock {\em Singular points of complex hypersurfaces}.
\newblock Annals of Mathematics Studies, No. 61. Princeton University Press,
  Princeton, N.J., 1968.

\bibitem{OkaVar}
M.~Oka.
\newblock Mixed functions of strongly polar weighted homogeneous face type,
  arxiv 1202.2166v1.

\bibitem{OkaPolar}
M.~Oka.
\newblock Topology of polar weighted homogeneous hypersurfaces.
\newblock {\em Kodai Math. J.}, 31(2):163--182, 2008.

\bibitem{OkaMix}
M.~Oka.
\newblock Non-degenerate mixed functions.
\newblock {\em Kodai Math. J.}, 33(1):1--62, 2010.

\end{thebibliography}
\end{document}